\documentclass[10pt]{article}
\usepackage{amsmath,amssymb,amsthm,latexsym}

\newtheorem{theorem}{Theorem}[section]
\newtheorem{proposition}[theorem]{Proposition}
\newtheorem{lemma}[theorem]{Lemma}
\newtheorem{corollary}[theorem]{Corollary}
\theoremstyle{definition}
\newtheorem{definition}[theorem]{Definition}

\theoremstyle{remark}
\newtheorem{remark}[theorem]{Remark}

\newcommand{\N}{\mathbb{N}}
\newcommand{\Z}{\mathbb{Z}}

\newcommand{\R}{\mathbb{R}}
\newcommand{\C}{\mathbb{C}}

\newcommand{\ain}{{a_{i,n}}}
\newcommand{\ajn}{{a_{j,n}}}
\newcommand{\tA}{\tilde{A}}
\newcommand{\Acentral}{A_{\omega}}

\newcommand{\Aomega}{{A^{\omega}}}
\newcommand{\Apositiveone}{{A_{+, \|\cdot\|}}}
\newcommand{\Atomega}{{A_{{\rm t}, \omega}}}
\newcommand{\Asecond}{{A^{**}}}
\newcommand{\alphabar}{\overline{\alpha}}
\newcommand{\alphaklm}{\alpha_{k,l,m}}
\newcommand{\talpha}{{\widetilde{\alpha}}}

\newcommand{\bin}{{b_{i,n}}}

\newcommand{\limn}{\displaystyle \lim_{n\to\infty}}
\newcommand{\limomega}{\displaystyle \lim_{n\to\omega}}
\newcommand{\liml}{\displaystyle \lim_{l\to\infty}}
\newcommand{\limlambda}{\displaystyle \lim_{\lambda}}

\newcommand{\linfty}{\ell^{\infty}}
\newcommand{\lNn}{{l_{N, n}}}

\newcommand{\cphi}{c_{\varphi}}

\newcommand{\elambda}{{e_{\lambda}}}
\newcommand{\eii}{{e_{i,i}}}
\newcommand{\eij}{{e_{i,j}}}
\newcommand{\ein}{{e_{i,n}}}
\newcommand{\ejn}{{e_{j,n}}}
\newcommand{\einl}{{e_{i,n,l}}}
\newcommand{\einln}{e_{i,n,l_n}}
\newcommand{\fil}{{f_{i,l}}}
\newcommand{\filn}{{f_{i,l_n}}}
\newcommand{\fjln}{{f_{j,l_n}}}
\newcommand{\flambda}{{f_{\lambda}}}
\newcommand{\forany}{{\rm\ for\ any\ }}
\newcommand{\Iomega}{{I_{\omega}}}
\newcommand{\Js}{\mathcal{Z}}
\newcommand{\Jtau}{J_{\tau}}
\newcommand{\T}{\ensuremath{\mathbb{T}}}
\newcommand{\Cs}{$\mathrm{C}^*$-al\-ge\-bra}
\newcommand{\Css}{$\mathrm{C}^*$-sub\-al\-ge\-bra}

\newcommand{\cR}{\mathcal{R}}

\newcommand{\cM}{{\mathcal{M}}}

\newcommand{\dt}{d_{\tau}}
\newcommand{\dextreme}{\partial_{\rm e}}
\newcommand{\te}{{\tilde{e}}}
\newcommand{\tfn}{{\tilde{f}_n}}
\newcommand{\TA}{{T(A)}}

\newcommand{\ulambda}{{u_{\lambda}}}
\newcommand{\uEel}{{u_{E, \varepsilon, \lambda}}}
\newcommand{\varphilambda}{{\varphi_{\lambda}}}
\newcommand{\varphiomega}{{\varphi}^{\omega}}
\newcommand{\varphiomegalambda}{{\varphi}_{\lambda}^{\omega}}
\newcommand{\varphiomegalambdan}{{\varphi}_{\lambda_n}^{\omega}}
\newcommand{\varphitaun}{{\varphi}_{\tau, n}}
\newcommand{\xlambda}{{x_{\lambda}}}

\newcommand{\yln}{{y_{l,n}}}
\newcommand{\ylnn}{{y_{l_n,n}}}

\DeclareMathOperator{\id}{id}
\DeclareMathOperator{\Ad}{Ad}

\begin{document}
\title{Actions of amenable groups and  crossed products  of $\Js$-absorbing \Cs s}
\author{
Yasuhiko Sato \thanks{The author was supported by JSPS KAKENHI, Grant Number JP 15K17553, JSPS Overseas Research Fellowships,  and Department of Mathematics at  Purdue University.}\\
}
\date{}

\maketitle
\begin{abstract} 
We study actions of countable discrete amenable groups on unital separable simple nuclear $\Js$-absorbing \Cs{}s. 
Under a certain assumption on tracial states, which is automatically satisfied  in the case of a unique tracial state, the crossed product is shown to absorb the Jiang-Su algebra $\Js$ tensorially. 
\end{abstract}

\section{Introduction}\label{Sec1}

In 1993, O. Bratteli, D. E. Evans, and A. Kishimoto \cite{BEK} determined a concrete structure of crossed products for outer Bogoliubov (quasi-free) actions of $\Z$. They showed that those crossed products of the CAR algebra are divided in two cases. In the case that the action is strongly outer, they showed that it has the Rohlin property and the crossed product is isomorphic to a simple inductive limit of $C(\T)\otimes M_k$, called an AT-algebra \cite{Ell}, of real rank zero with a unique tracial state. On the other hand, if the action is not strongly outer but still outer, then the crossed product is isomorphic to a simple inductive limit of subalgebras of $C(\T)\otimes M_k$. In that case the extreme boundary of trace space is $\T$. It was not known at that time that the crossed product of outer actions can be classified, however the recent classification theorem \cite{LN}, \cite{Win}, \cite{EGLN}, \cite{TiWW} shows that both cases are classifiable within the Elliott program \cite{ET}. By this reason, we now notice that in both cases these crossed products absorb the Jiang-Su algebra tensorially. This paper is inspired by the results from \cite{BEK} led us to explore the classifiability of crossed products of outer actions.  

The following is the main result of this paper. 
\begin{theorem}\label{ThmMain}
Let $G$ be a countable discrete amenable group,  $A$ a unital separable simple nuclear \Cs, and $\alpha$ an action of $G$ on $A$.  Assume that $A$ absorbs the Jiang-Su algebra $\Js$ tensorially, (i.e., $A\otimes \Js \cong A$), the extreme boundary of the trace space of $A$ is compact and finite dimensional and that $\alpha$ fixes any tracial state of $A$.
Then the crossed product \Cs{} $A\rtimes_{\alpha} G$ also absorbs the Jiang-Su algebra tensorially.
\end{theorem}

Combining this with \cite[Theorem F]{BBSTWW}, we can estimate the nuclear dimension of crossed products assuming the trace space is a Bauer simplex (i.e., its
extreme boundary is compact). 

\begin{corollary}Let $G$ be a countable discrete amenable group, $A$ a unital separable simple nuclear $\Js$-absorbing \Cs\ with a unique tracial state, and $\alpha$ an outer action of $G$ on $A$. If the trace space of $A\rtimes_{\alpha} G$ is a Bauer simplex, then the nuclear dimension of $A\rtimes_{\alpha} G$ is at most one. 
\end{corollary}

The Jiang-Su algebra $\Js$ is constructed as a unital separable simple infinite-dimensional nuclear \Cs\ with a unique tracial state whose K-theoretic invariants are same as that of the complex numbers,  \cite{JS}. In the current classification theorem of \Cs s, the absorption of $\Js$ is regarded as one of the regularity properties of classifiable \Cs s. Actually, A. S. Toms and W. Winter conjectured that $\Js$-absorption is equivalent to other regularity properties, such as strict comparison and finiteness of nuclear dimension, for unital separable simple infinite-dimensional nuclear \Cs s, \cite{Win1}, \cite{Win2}. At the present stage, we know that this conjecture holds true under the tracial condition in Theorem \ref{ThmMain}, and hence in the case of unique tracial state, \cite{Ror}, \cite{Win2}, \cite{MS3}, \cite{Sat1}, \cite{KR}, \cite{TWW}, \cite{BBSTWW}. In the same direction, the main theorem focusses on the  classifiability of crossed products for outer actions. 

The first remarkable result on $\Js$-absorbing crossed products was obtained by I. Hirshberg and W. Winter in \cite{HW}. They showed the permanence property of $\Js$-absorption  under taking crossed products by $\Z$, $\R$, and compact groups, assuming the Rohlin property. More generally, they showed such a result for $D$-stability of strongly self-absorbing K$_1$-injective \Cs\ $D$. In a joint work with H. Matui \cite{MS2}, we proved $\Js$-absorption for twisted crossed products of elementary amenable groups assuming strong outerness and some mild conditions for given amenable \Cs s. Introducing the notion of Rokhlin dimension for a certain class of groups, Hirshberg, Winter and J. Zacharias also showed a permanence result of $\Js$-absorption for crossed products under finiteness of Rokhlin dimension \cite{HWZ}.   The main result of this paper widens the study of $\Js$-absorbing crossed products in the directions of amenable groups and actions without outerness.

\bigskip 

\noindent Before ending the introduction, let us fix some notations. 
Throughout this paper, we let $\omega$ be a free ultrafilter on the set of natural numbers $\N$. For a separable \Cs{} $A$, we denote by $A^{\omega}$ the ultraproduct \Cs\ defined by  
\[ \Aomega= \linfty(\N, A)/\left\{(a_n)_n\in \linfty (\N, A)\ |\ \limomega\left\lVert a_n \right\rVert =0\right\}.\]
Regarding $A$ as the \Css\ of $\Aomega$ consisting of all constant sequences, we define the central sequence algebra by $\Acentral = A' \cap \Aomega$. 
A sequence $(a_n)_n\in \linfty(\N, A)$ is called a central sequence, if it is a representative of an element in $\Acentral$.

For an automorphism $\alpha$ of $A$, we can consider its standard extension on $\Aomega$ and $\Acentral$. We denote it by the same symbol $\alpha$. 
We denote by $\TA$ the convex set of all tracial states of $A$ and $\dextreme(\TA)$ the extreme boundary of $\TA$. 

\section{A small Rohlin tower}\label{Sec2}
We start with a characterization of outerness by a small Rohlin tower, which is based on \cite[Lemma 1.1]{Kis1}. In the case of simple purely infinite \Cs s, this result is known to be obtained by a tower of projections, \cite[Lemma 3]{Nak}, \cite[Lemma 3.4]{IM}. 

\begin{proposition} \label{PropOuter}
Let $A$ be a separable simple \Cs, and $\{\alpha_i\}_{i\in \N}$ a countable set of automorphisms of $A$.  Then  $\alpha_i$ is outer for any $i\in\N$, if and only if there exists a net $(\flambda)_{\lambda\in\Lambda}$ of positive contractions in $A$ such that 
\[\limlambda \left\lVert[\flambda, a]\right\rVert =0,\quad \limlambda\left\lVert \alpha_i (\flambda)\flambda\right\rVert=0,\] 
for any $a\in A$ and $i\in\N$, and $f_{\lambda}$ converges to a central projection of $\Asecond$ strongly.
\end{proposition}

The positive element in the above proposition can be derived from a pure state in the next theorem. This is just a variation of \cite[Theorem 2.1]{Kis1} for countable outer automorphisms.

\begin{theorem}\label{ThmPureState}
Let $A$ be a separable simple \Cs{} and $\{\alpha_i\}_{i\in\N}$ a set of outer automorphisms of $A$. Then there exists a pure state $\varphi$ of $A$ such that $\varphi\circ \alpha_i$ is disjoint from $\varphi$ for any $i\in \N$.
\end{theorem}

\begin{proof}
(This argument is essentially same as the proof of \cite[Theorem 2.1]{Kis1})
Let $\tA$ be the unitization of $A$ and $\{ u_n\}_{n\in\N}$ be a dense subset of the unitary group of $\tA$. Set $\sigma_n =\Ad u_n$ for $n\in\N$ and $\varepsilon \in (0, 1/2)$. We let $\Phi : \N \rightarrow \N\times \N$ be an isomorphism as sets. 

For $(i,j)\in \N\times\N$ with $(i,j)=\Phi(1)$, applying \cite[Lemma 1.1]{Kis1} to $A$ and $\sigma_i\circ\alpha_j$ we can obtain a positive element $f_1$ in $A$ such that $\|f_1\|=1$, $\|f_1\sigma_i\circ\alpha_j(f_1) \|<\varepsilon$. Taking a slight perturbation of $f_1$ we may find a positive element $a_1$ in $A$ such that $\|a_1\|=1$, $a_1f_1=a_1$. For $(i,j)\in \N\times \N$ with $(i,j)=\Phi(2)$, applying \cite[Lemma 1.1]{Kis1} to $\overline{a_1Aa_1}$ and $\sigma_i\circ\alpha_j$, we obtain positive elements $f_2$, $a_2$ in $\overline{a_1Aa_1}$ such that $\| f_2\|=\|a_2\|=1$, $a_2f_2= a_2$, and $\|f_2\sigma_i\circ\alpha_j(f_2)\|<\varepsilon$ for $(i,j)=\Phi(2)$. Note that $f_2f_1=f_2$. 

Repeating this argument inductively along with the order of $\Phi(n)$, $n\in\N$, we obtain sequences  $(f_n)_n$ and $(a_n)_n$ of positive elements in $\overline{a_{n-1}Aa_{n-1}}$ such that  $f_nf_{n-1}=f_n$, 
\[ \|f_n\|=\|a_n\| =1,\quad a_nf_n =a_n,\quad \|f_n\sigma_i\circ\alpha_j(f_n)\|<\varepsilon \quad{\rm for\ }(i,j)=\Phi(n).\]

We let $S$ be the set of all states $\psi$ of $A$ satisfying $\psi(f_n)=1$ for any $n\in\N$. Then it follows that $S$ is a non-empty closed face of the state space. The desired pure state $\varphi$ is obtained as an extremal point of $S$. Actually if $\varphi\circ\alpha_{j'}$ is not disjoint from $\varphi$, then the irreducible representations of $\varphi\circ\alpha_{j'}$ and $\varphi$ are spatially equivalent. Then Kadison's transitivity theorem allows us to get a unitary $u$ in $\tA$ such that $\varphi\circ\alpha_{j'}\circ\Ad u = \varphi $. Since $\{u_n\}_{n\in\N}$ is dense, there exists $i'\in\N$ such that $\|\varphi\circ\sigma_{i'}\circ\alpha_{j'} -\varphi\| < \varepsilon$, which contradicts to $\| f_n \sigma_{i'}\circ\alpha_{j'} (f_n)\| <\varepsilon$ for $n\in\N$ with $\Phi(n)=(i',j')$, because 
\[1=\varphi(f_n) \approx_{\varepsilon} \varphi \circ\sigma_{i'}\circ\alpha_{j'} (f_n) = \varphi(f_n\sigma_{i'}\circ\alpha_{j'}(f_n)) \approx_{\varepsilon} 0.\]

\end{proof}

\begin{proof}[Proof of Proposition \ref{PropOuter}]
Assume that $\{\alpha_i\}_{i\in\N}$ is a set of outer automorphisms.
By Theorem \ref{ThmPureState}, there exists a pure state $\varphi$ of $A$ such that $\varphi\circ\alpha_i$ is disjoint from $\varphi$ for any $i\in\N$. We denote by $\overline{\alpha}_i$ the extension of $\alpha_i$ on the enveloping  von Neumann algebra $\Asecond$, and $\cphi \in \Asecond$ the central cover of the irreducible representation associated with $\varphi$, then it follows that $\overline{\alpha}_i(\cphi)\cphi =c_{\varphi\circ\alpha_i^{-1}}c_{\varphi}=0$ for any $i\in\N$. Applying \cite[Lemma 1.1]{AP} we can obtain a bounded net $\xlambda \in A$, $\lambda \in \Lambda$ which converges to $\cphi$ strongly in $\Asecond$ and $\limlambda \| [x_{\lambda}, a] \| =0$ for any $a\in A$.
In particular, since $\cphi$ is a projection, the argument in the proof of \cite[Lemma 1.1]{AP} allows us to see $\xlambda$ are  positive contractions in $A$. Note that $\alpha_i(\xlambda)\xlambda$ converges to 0 strongly for all $i\in \N$.

Let $F$ be a finite subset of $\N$ and $\varepsilon > 0$. By a standard reindex argument it is enough to obtain a net $(\flambda)_{\lambda}$ of positive contractions in $A$ such that $\limlambda \left\lVert [f_{\lambda}, a]\right\rVert =0$ for any $a\in A$, $\limlambda \| \alpha_i (\flambda)\ \flambda\| <\varepsilon$ for any $i\in F$, and $(\flambda)_{\lambda}$ converges to $c_{\varphi}$ strongly in $\Asecond$.

Let $m\in \N$ be such that $1/m < \varepsilon$ and $g_m$ the continuous function on $[0, \infty)$ defined by $g_m (t) =\min \{ 1, mt\}$. We define positive elements $\elambda$, $\flambda$ in $A$, $\lambda\in \Lambda$ by 
\[ \elambda = \xlambda^{1/2}\left(\sum_{i\in F} \alpha_i(\xlambda)\right)\xlambda^{1/2},\quad \flambda = \xlambda^{1/2} (1- g_m(\elambda))\xlambda^{1/2}.\]
Then it follows that $\limlambda \| [\elambda, a] \| + \| [\flambda, a]\| =0$ for any $a\in A$, and $\elambda$, $g_m(\elambda)$ converge to 0 strongly, see \cite[Proposition 2.3.2]{Ped} for example. Hence we have $\flambda \to \cphi$ strongly. 

By the definition of $g_m$ it follows that for $i\in F$
\begin{align*}
\|\alpha_i(\flambda)\flambda \|^2 
&\leq \left\lVert\flambda\left(\sum_{i\in F} \alpha_i (\xlambda)\right)\flambda \right\rVert \\ 
&\leq \left\lVert (1- g_m (\elambda)) \elambda (1-g_m (\elambda))\right\rVert < 1/m.
\end{align*} 
A similar argument can be found in the proof of \cite[Theorem 4.5]{Kis2} and \cite[Lemma 3.2]{MS3}.

The converse direction is obvious. Actually, if there exists $i\in\N$ and a unitary $u$ in the unitization of $A$ such that $\alpha_i=\Ad u$, 
then $\alpha_i(\flambda)\flambda$ converges to a central projection in $\Asecond$ strongly. This is a contradiction.
\end{proof}

\section{ G-equivariant large embeddability of cones}
With the benefit of Connes' fundamental theorem for injective factors \cite{Con}, we have obtained tracially large embeddings of cones into ultraproducts of UHF algebras in \cite[Proposition 3.2]{SWW}. In this section, we show a $G$-equivariant version of it for amenable groups as follows.  

 For $\tau\in \TA$, we define the trace kernel ideal by 
\[\Jtau=\left\{ (a_n)_n \in \linfty (\N, A)\ |\ \limomega \tau(a_n^*a_n)=0\right\},\]
and we define
\[\Iomega=\left\{ (a_n)_n\in\linfty (\N, A)\ |\ \limomega \sup_{\tau\in \TA}\tau(a_n^*a_n)=0\right\}.\]
Since $\{(a_n)_n\in\linfty (\N, A)\ |\ \limomega \|a_n\|=0\}$ is contained in $\Iomega \subset \Jtau$, we regard $\Jtau$ and $\Iomega$ as ideals of $\Aomega$. We set $\Atomega= \Acentral/ (\Iomega\cap A' )$. 
We denote by $\id $ the identity map on $(0,1]$. 

\begin{proposition}\label{PropEmbedCone}
Let $G$ be a countable discrete amenable group, $A$ and $B$ unital separable  simple infinite-dimensional nuclear \Cs{}s  with at least one tracial state, and $\alpha$ an action of $G$ on $A$. Suppose that  an extremal tracial state $\tau$ of $A$ is a fixed point of $\alpha$. Then there exists an embedding $\varphi$ of  $C_0((0,1])\otimes B$ into $\Acentral$ such that 
\[\alpha_g (\varphi(x))=\varphi(x)\ \forany g\in G,   x\in C_0((0,1])\otimes B,\quad {\rm and\quad } 1-\varphi(\id \otimes 1_B)\in \Jtau.\]

\end{proposition}

The embedding in the above proposition shall be constructed based on the following technique in the study of von Neumann algebras. This observation is contained in the proof of \cite[Theorem 3.1]{MK} as an application of Ocneanu's result \cite[Lemma 8.3]{Ocn}. For the reader's convenience, we revisit this argument here.

We let $\cR$ be the injective II$_1$-factor, and $\cM$ be the $\omega$-central sequence algebra of $\cR$ which is defined by 
\[ \cM=\left(\linfty (\N, \cR) /\left\{(x_n)_n\ |\ \limomega \tau_{\cR} (x_n^*x_n)=0\right\}\right)\cap \cR',\]
where $\tau_{\cR}$ is the unique trace on $\cR$ and $\cR$ is regarded as the subalgebra consisting of constant sequences. For an automorphism $\alpha$ on $\cR$, we denote the extension on $\cM$ by the same symbol $\alpha$.

\begin{theorem}\label{ThmMK}
Let $G$ be a countable discrete amenable group, and $\alpha$ an action of $G$ on the injective II$_1$-factor $\cR$. Then for any $k\in\N$ there exists a unital embedding of $M_k$ into the $\alpha$-fixed point subalgebra of $\cM$.
\end{theorem}

\begin{proof}
Let $K$ be the normal subgroup of $G$ consisting of all $g\in G$ such that $\alpha_g(x)=x$ for any $x\in \cM$, and let $H=G/K$. Note that $H$ is also amenable. We define an action $\talpha$ of $H$ on $\cM$ by $\talpha_{[g]} (x) = \alpha_g (x)$ for $x\in \cM$, where $[g]\in H$ is the equivalence class of $g\in G$. Then it follows that $\talpha_{[g]}$ is non-trivial in $\cM$ for any $[g]\in H\setminus \{ 1\}$. 

Remark that for $[g]\in H\setminus \{1\}$, $\talpha_{[g]}$ is liftable as an automorphism on $\cR$ and then $\talpha_{[g]}$ is strongly outer in $\cM$, \cite[Lemma 5.7]{Ocn}. Applying \cite[Lemma 8.3]{Ocn} to $\talpha$ we see that the $\talpha$-fixed point subalgebra of $\cM$ is a type II$_1$ von Neumann algebra, which contains a matrix algebra $M_k$ unitally. The fixed point algebra of $\talpha$ in $\cM$ is nothing but of $\alpha$, then we obtain the desired unital embedding of $M_k$.  
\end{proof}

The following lemma is known to be generalized to non-amenable groups \cite[1.4]{Kas}. Here we give a proof of the easier case for amenable groups.  

\begin{lemma}\label{LemQuasiCent}
Let $G$ be a discrete amenable group, and $A$ a \Cs{} with an action $\alpha$ of $G$. For any $\alpha$-invariant ideal $I$ of $A$, there exists a quasicentral approximate identity $\ulambda\in I$, $\lambda\in \Lambda$ of $I$ in $A$ which is asymptotically $\alpha$-fixed: 
\[ \limlambda \left\lVert \alpha_g (\ulambda)- \ulambda \right\rVert =0\quad\forany g\in G.\]
\end{lemma}
\begin{proof}
For any finite subset $E$ of $G$ and $\varepsilon >0$, we obtain a F\o lner set $F$ of $G$ which means 
\[ \max_{g\in E} \frac{|(gF\cup F) \setminus (gF\cap F)|}{|F|} < \varepsilon.\]
Let $\ulambda \in I$, $\lambda\in\Lambda'$ be a quasicentral approximate identity of $I$ in $A$. 
Set $\displaystyle \uEel = \frac{1}{|F|}\sum_{g\in F}\alpha_g (\ulambda)\in I$. Then $\uEel$, $\lambda\in \Lambda'$ is also a quasicentral approximate identity of $I$ which satisfies 
\[ \left\lVert \alpha_g(\uEel) -\uEel \right\rVert = \frac{1}{|F|} \left\lVert\sum_{h\in F} \alpha_{gh}(\ulambda) - \alpha_h(\ulambda)\right\rVert <\varepsilon,\]
for any $g\in E$ and $\lambda\in \Lambda'$. Then for finite subsets $A_0\subset A$ and $I_0\subset I$, there exists $\lambda_0\in\Lambda'$ such that $\left\lVert[u_{E,\varepsilon,\lambda_0}, a]\right\rVert <\varepsilon$, $\left\lVert u_{E,\varepsilon, \lambda_0} b - b \right\rVert<\varepsilon$ for any $a\in A_0$ and $b\in I_0$. Define $u_{E, A_0, I_0, \varepsilon}= u_{E, \varepsilon, \lambda_0}$. 
Regarding the set of all $(E, A_0, I_0, \varepsilon)$ as a net with the natural ordering, we see that $u_{E, A_0, I_0, \varepsilon}$ is a quasicentral approximate identity which is asymptotically $\alpha$-fixed.
\end{proof}

\begin{proof}[Proof of Proposition \ref{PropEmbedCone}]
We let $\pi_{\tau}$ be the GNS-representation associated with $\tau$. 
By $\tau\circ \alpha_g =\tau$ for any $g\in G$, we can consider the extension $\alphabar_g$ on $\pi_{\tau} (A)''$ which is determined by $\alphabar_g \circ\pi_{\tau}(a) = \pi_{\tau} \circ \alpha_g (a)$ for $a\in A$.  Because $\pi_{\tau} (A)''$ is isomorphic to $\cR$ (see \cite{Con}), we regard $\alphabar$ as an action of $G$ on $\cR$. Note that $\Jtau$ is an $\alpha$-invariant ideal of $\Aomega$ and that $\pi_{\tau}$ induces the isomorphism $\Phi : \Acentral / (\Jtau \cap A') \rightarrow \cM$. We denote by $\alpha'$ the canonical action on $\Acentral / (\Jtau \cap A')$ induced by $\alpha$. Since $\alphabar_g\circ\Phi = \Phi\circ {\alpha'}_g$, we identify $\alpha'$ with $\alphabar$ on $\cM$. By Theorem \ref{ThmMK}, we obtain a unital embedding $\Phi$ of $\cR$ into $\cM$ such that 
\[ \alphabar_g\circ\Phi (x) =\Phi (x)\quad \forany x\in \cR,\ g\in G.\]
 
Let $\tau_B$ be an extremal trace of $B$. Then we have a unital embedding $\Phi'$ of $\pi_{\tau_B}(B)'' (\cong \cR)$ into the $\alpha'$-fixed point subalgebra of  $\Acentral / (\Jtau\cap A') (\cong \cM)$.
By the Choi-Effros lifting theorem \cite{CH}, we obtain a unital completely positive map $\varphiomega : B \rightarrow \Acentral$ such that 
\[ \varphiomega (x) + (\Jtau \cap A') =\Phi'(x)\quad\forany x\in B.\]

By Lemma \ref{LemQuasiCent}, there exists a quasicentral approximate identity $\ulambda\in \Jtau\cap A'$, $\lambda\in\Lambda$ which is asymptotically $\alpha$-fixed. We define  a completely positive contraction $\varphiomegalambda$ by 
\[\varphiomegalambda (x) = (1-\ulambda)^{1/2}\varphiomega (x) (1-\ulambda)^{1/2}, \quad \lambda\in \Lambda,\ x\in B.\] 
Note that $1-\varphiomegalambda (1_B) \in \Jtau $ for any $\lambda\in\Lambda$.
Since $B$ is separable, and $G$ is countable, there exists an increasing sequence $\lambda_n\in \Lambda$, $n\in\N$ such that
\begin{align*}
&\limsup_{n\to\infty} \left\lVert \varphiomegalambdan (x)\cdot\varphiomegalambdan (y)\right\rVert \leq\left\lVert xy\right\rVert \quad\forany x, y\in B, \\
&\limn \left\lVert \alpha_g(\varphiomegalambdan(x)) -\varphiomegalambdan(x) \right\rVert =0 \quad \forany g\in G\ {\rm and \ } x\in B.
\end{align*}
Then the standard reindex argument allows us to get an order zero (disjointness preserving) completely positive contraction $\varphi$ from $B$ into $\Acentral$ such that $\alpha_g(\varphi(x)) =\varphi(x)$ for any $g\in G$ and $1-\varphi(1_B)\in \Jtau$. Because any order zero completely positive contraction can be identified with a $*$-homomorphism from the cone over the domain, see \cite[Proposition 1.2]{Win1}, we obtain the required embedding of $C_0((0,1])\otimes B$ into $\Acentral$ satisfying the required conditions. 
\end{proof}

The following Lemmas are $G$-equivariant versions of \cite[(ii) Lemma 4.2, Proposition 5.1]{Sat1}. In analogous ways, it is not so hard to check the detail. 
\begin{lemma}\label{LemOrthogonal}
Let $G$, $A$, $\alpha$ be as in Proposition \ref{PropEmbedCone}. Suppose that $\dextreme (\TA)$ is compact and $\alpha$ fixes any tracial state of $A$. Then, for any mutually orthogonal positive functions $f_i\in C(\dextreme (\TA)$, $i=1,2,...,N$, there exist sequences $(\ain)_n$, $i=1,2,...,N$ of positive elements in $A$ such that $(\ain)_n\in\Acentral$ for any $i=1,2,...,N$, $(\ain \ajn)_n=0$ in $\Acentral$ for $i\neq j$, and 
\[ \limomega\max_{\tau\in \dextreme (\TA) } \left\lvert \tau(\ain) - f_i(\tau)\right\rvert =0,\quad (\alpha_g(\ain))_n = (\ain)_n\ {\rm in}\ \Acentral\ \forany\ g\in G,\ i=1,2,...,N.\] 
\end{lemma}
\begin{proof}
Because of \cite[Proposition 4.1]{Sat1}, which is based on \cite[Lemma 4.4]{DT}, we obtain sequences $(\bin)_n$, $i=1,2,...,N$ of positive elements in $A$ such that $\|\bin\|\leq\|f_i\|$ and 
\[ \limomega \max_{\tau\in\dextreme(\TA)} \left\lvert \tau(\bin) -f_i(\tau)\right\rvert =0.\] 
Because $A$ is nuclear, we may assume that $(\bin)_n\in \Acentral$ (see \cite[Corollary 3.3]{Sat1} for example).

Since $\alpha$ fixes any $\tau\in\dextreme (\TA)$ and $G$ is a countable amenable group, using a similar method in the proof of Lemma \ref{LemQuasiCent}, we may further assume that $(\alpha_g(\bin))_n= (\bin)_n$ in $\Acentral$ for any $g\in G$ and $i=1,2,...,N$. 

By a similar argument in the proof of Proposition \ref{PropOuter}, we obtain sequences $(\ain)_n$, $i=1,2,...,N$ of positive elements in $A$ such that $(\ain)_n\in \Acentral$, $\ain\leq \bin$, $(\alpha_g(\ain))_n=(\ain)_n$ in $\Acentral$ for any $g\in G$, $i=1,2,...,N$, $(\ain\ajn)_n=0$ in $\Acentral$ for $i\neq j$, and 
\[\limomega\max_{\tau\in\dextreme(\TA)}\left\lvert\tau(\ain-\bin)\right\rvert =0.\]
\end{proof}

\begin{lemma}\label{LemEmbedMk}
Let $G$, $A$, $\alpha$ be as in Proposition \ref{PropEmbedCone}. Assume that $\dextreme(\TA)$ is compact and finite dimensional and $\alpha$ fixes any tracial state of $A$. Then for any $k\in \N$, there exists a unital embedding of $M_k$ into the $\alpha$-fixed point subalgebra of $\Atomega$.
\end{lemma}

\begin{proof}
For a natural number $k\in \N\setminus \{1\}$ and an $\alpha$-fixed  $\tau\in \dextreme (\TA)$, applying Proposition \ref{PropEmbedCone} to the UHF algebra $M_{k^{\infty}} (=B)$  we obtain a sequence of completely positive contractions $\varphitaun : M_k \rightarrow A$, $n\in\N$ satisfying the same conditions in the proof of \cite[Proposition 5.1]{Sat1}. Additionally, now these $\varphitaun$, $n\in\N$ satisfy
\[(\alpha_g(\varphitaun (x)))_n = (\varphitaun(x))_n\quad {\rm in \ } \Acentral,\]
for any $g\in G$ and $x\in M_k$. Then the same argument in the proof of \cite[Proposition 5.1]{Sat1} and Lemma \ref{LemOrthogonal} allows us to get the required embedding. 

\end{proof}

\begin{corollary}\label{CorEmbedding}
Let $G$, $A$, $\alpha$ be as in Proposition \ref{PropEmbedCone}. Assume that $\dextreme (\TA)$ is compact and finite dimensional and $\alpha$ fixes any tracial state of $A$. Then for any $k\in\N$ there exists an embedding $\varphi$ of $C_0((0,1])\otimes M_k$ into the $\alpha$-fixed point subalgebra of $A_{\omega}$ such that $1-\varphi(\id\otimes 1_k)\in \Iomega$.
\end{corollary}
\begin{proof}
Since $\Iomega$ is $\alpha$-invariant ideal, the proof follows from the same method of Proposition \ref{PropEmbedCone}. Actually, by Lemma \ref{LemEmbedMk} we obtain a unital completely positive contraction $\varphi$ from $M_k$ into $\Acentral$ such that $\varphi(x)\varphi(y) -\varphi(xy)\in \Iomega\cap A'$, $\alpha_g(\varphi(x)) -\varphi(x)\in \Iomega\cap A'$ for any $x, y\in M_k$, $g\in G$. Let $\ulambda$, $\lambda\in\Lambda$ be a quasi-central approximate identity of $\Iomega\cap A'$ in $\Acentral$ which is asymptotically $\alpha$-fixed. We define completely positive contraction $\varphilambda : M_k \rightarrow \Acentral$ by $\varphilambda (x) =(1-\ulambda)^{1/2} \varphi(x) (1-\ulambda)^{1/2}$ for $\lambda\in\Lambda$, $x\in M_k$. Then it follows that
\begin{align*}
&\lim_{\mu}\sup_{\mu\leq\lambda} \left\lVert\varphilambda(x)\varphilambda(y)\right\rVert \leq \left\lVert xy \right\rVert \quad\forany x, y\in M_k, \\
&\lim_{\lambda} \left\lVert\alpha_g(\varphilambda (x)) -\varphilambda(x) \right\rVert =0, \quad\forany g\in G, x\in M_k.
\end{align*}
Since $A$ and $M_k$ are separable, and $G$ is countable,  we can obtain an order zero completely positive contraction $\varphi : M_k\rightarrow\Acentral$ such that $\alpha_g(\varphi(x))=\varphi(x)$ for any $g\in G$, $x\in M_k$, and $1-\varphi(1_k)\in \Iomega$, which induces the embedding of $C_0((0,1])\otimes M_k$.  
\end{proof}

\section{ Property of tiny isometries}

This section is a continuation of our previous works \cite{MS3}, \cite{Sat1}, in which the absorption of the Jiang-Su algebra was shown by using a technical condition named property (SI). The property of small isometries (SI) means, loosely speaking, the existence of a sort of small isometries in the central sequence algebra, under the assumption of the strict comparison. In this section, we further promote the existence of much smaller elements in the following sense. 

In the rest of this paper, we denote by $\Apositiveone$ the set of all norm one positive elements in a \Cs\ $A$. For $m\in\N$, 
we let $h_m$ be the continuous function on $[0,1]$ defined by 
\[ h_m(t) =\max\left\{ 0,\ m \left(t-1+m^{-1}\right)\right\},\quad t\in[0,1].\]

\begin{definition}\label{DefTI}

Let $A$ be a separable \Cs\ with $\TA\neq \emptyset$. We say that $A$ has {\it property {\rm (TI)}}: if for any sequences $(e_n)_n$, $(f_n)_n$ of positive contractions in $A$ and for a countable dense subset $\{ a_k\}_{k\in\N}$ in $\Apositiveone$ satisfying $(e_n)_n$, $(f_n)_n\in \Acentral$ and 
\[ \limomega\ \frac{\max_{\tau\in \TA}\tau(e_n)}{\min_{\tau\in \TA} \tau\left(h_m\left(f_n^{1/2}a_kf_n^{1/2}\right)\right)}=0\quad \forany\ k, m\in \N,\]
there exists a sequence $(s_n)_n$ in $A$ such that $(s_n)_n\in \Acentral$, 
\[ (s_n^*s_n)_n=(e_n)_n,\quad (f_ns_n)_n=(s_n)_n\ {\rm in \ }\Acentral.\]
\end{definition}

\begin{remark} 
If $A$ is simple and $(e_n)_n\neq 0$ in $\Acentral$, then the above condition for $(e_n)_n$, $(f_n)_n$, $\{a_k\}_k$ automatically implies $\limomega\left\lVert f_n \right\rVert =1$. 
Actually, for any $k$, $m\in\N$, when $n$ is in a neighborhood of $\omega$ it follows that 
\[ 0<\max_{\tau\in\TA} \tau(e_n) < \min_{\tau\in\TA} \tau(h_m (f_n^{1/2} a_k f_n^{1/2})),\]
which implies $h_m(f_n^{1/2} a_k f_n^{1/2})\neq 0$, then $\left\lVert f_n^{1/2} a_k f_n^{1/2} \right\rVert \geq 1-m^{-1}$. 
\end{remark}

\begin{remark}
For a unital simple \Cs\ $A$, we notice that property (TI) is stronger than (SI) in \cite[Definition 4.1]{MS3}. Actually, if a sequence $(f_n)_n$ of positive contractions in $A$ satisfies $(f_n)_n\in \Acentral$ and 
\[ \lim_{m\to\infty} \limomega \min_{\tau\in \TA} \tau(f_n^m)>0,\]
then there exists a sequence $(\tfn)_n$ of positive contractions in $A$ such that $(\tfn f_n)_n=(\tfn)_n$, $(\tfn)_n\in \Acentral$, and $\limomega \min_{\tau} \tau(\tfn) >0$, (see \cite[Lemma 2.3.]{MS3} for example). Then, applying \cite[Lemma 2.4]{MS3} to $h_m(a_k)\neq 0$, we obtain $c>0$ such that 
\begin{align*}
\limomega\min_{\tau\in\TA} \tau (h_m(f_n a_k f_n)) &\geq \lim_n\min_{\tau} \tau(h_m(f_n a_k f_n)^{1/2} \tfn h_m(f_n a_k f_n)^{1/2}) \\
& \geq \lim_n\min_{\tau} \tau(h_m ( a_k )\tfn) \geq c \lim_n \min_{\tau} \tau(\tfn) >0. 
\end{align*}
Nevertheless, the next proposition shows that these two notions, (SI) and (TI), are equivalent and warrants (TI), under nuclearity and uniqueness of the tracial state. 
\end{remark}

\begin{proposition}\label{PropTI}
Let $A$ be a unital separable simple nuclear \Cs\ with compact finite dimensional $\dextreme (\TA)$. Then the following conditions are equivalent 
\begin{enumerate}
\item $A$ has strict comparison.
\item $A$ has property (SI) in the sense of \cite[Definition 4.1]{MS3}.
\item $A$ has property (TI).
\end{enumerate}
\end{proposition}

The definition of strict comparison for positive elements was given in \cite{Ror}, based on Blackadar's notion for projections \cite{Bl}. For $\tau\in \TA$ we denote by $\dt$ the dimension function which is defined by $\dt(a) =\limn \tau( a^{1/n})$ for a positive element $a\in A\otimes M_k$, here $\tau$ is identified with a unnormalized trace on $A\otimes M_k$. We say that a separable \Cs\ $A$ has {\it strict comparison } if for two positive elements $a$, $b$ $\in A\otimes M_k$ with $\dt (a) < \dt (b)$ for any $\tau\in\TA$ there exists a sequence $(r_n)_n$ in $A\otimes M_k$ such that $\limn r_n^* b r_n =a$ in the operator norm topology.

We have seen (i)$\Longleftrightarrow $ (ii) in \cite{Sat1}, \cite{TWW}, \cite{KR}, then the remaining assertion we should prove is (i)$\Longrightarrow$(iii), which shall be given in a similar way as (i)$\Longrightarrow$ (ii). The main difference between (i)$\Longrightarrow$(ii) and (i)$\Longrightarrow$(iii) is the following lemmas.

\begin{lemma}\label{LemCentral}
Let $A$ be a separable \Cs\ with $\TA\neq \emptyset$, and $(e_n)_n$ a sequence of  positive contractions in $A$. Then the following holds: 
\begin{enumerate}
\item If $(e_n)_n$ and countable sequences $x_{m,n}>0$, $m, n \in\N$ satisfy
\[ \limomega \frac{\max_{\tau\in\TA}\tau(e_n)}{x_{m,n}} =0 \quad\forany m\in\N,\]
then there exists a sequence $(\te_n)_n$ of positive contractions in $A$ such that $(\te_n)_n=(e_n)_n$ in $\Aomega$ and 
\[\limomega \frac{\max_{\tau\in \TA} \dt (\te_n)}{x_{m,n}} =0\quad \forany m\in \N.\]
\item Assume that there exists a $*$-homomorphism $\varphi : C_0((0,1])\otimes M_N \rightarrow \Acentral$ such that $1-\varphi(\id \otimes 1_N) \in \Iomega$, and that a sequence $x_n>0$, $n\in\N$, a subset $\{a_k\}_{k\in\N}$ of $\Apositiveone$, and $(e_n)_n$ satisfy $(e_n)_n \in \Acentral$ and 
\[ \limomega \frac{x_n}{\min_{\tau\in\TA} \tau\left(h_m\left({e_n}^{1/2} a_k {e_n}^{1/2}\right)\right)}=0, \forany k, m\in\N.\]
Then there exist sequences $(\ein)_n$, $i=1,2,...,N$ of positive contractions in $A$ such that $(\ein)_n\in \Acentral$,  $\ein\leq e_n$ for any $n\in\N$, $(\ein\ejn)_n=0$ in $\Acentral$ for $i\neq j$, and 
\[ \limomega \frac{x_n}{\min_{\tau\in\TA} \tau\left(h_m\left(\ein^{1/2} a_k \ein^{1/2}\right)\right)}=0, \forany k, m\in\N\ {\rm and}\ i=1,2,...,N.\]
\end{enumerate}
\end{lemma}
\begin{proof}
(i) By a diagonal argument, for $N\in\N$ it is enough to obtain $\te_n$ such that $(\te_n)_n =(e_n)_n$ in $\Aomega$ and 
\[ \limomega \max_{m=1,2,...,N}\frac{\max_{\tau\in \TA} \dt (\te_n)}{x_{m,n}}=0.\] 
Put $\displaystyle \varepsilon_n = \max_{\tau\in \TA} \tau(e_n)$ and $\displaystyle\delta_n =\min_{m=1,2,...,N} x_{m,n}>0$. Since $\varepsilon_n /\delta_n \to 0$, $n\to \omega$, there exists a sequence $r_n >0$, $n\in\N$ such that $r_n\to 0$ and $\varepsilon_n/(\delta_n r_n)\to 0$, $n\to\omega$. For $r>0$, we let $g_r$ be the continuous function defined by $g_r(t) =\max \{0, t-r\}$, $t\in [0,1]$. Then it follows that 
\[\frac{\max_{\tau\in \TA} \dt (g_{r_n} (e_n))}{\delta_n} \leq \frac{\varepsilon_n}{\delta_nr_n} \to 0,\quad n\to\omega.\]
Then a sequence $\te_n = g_{r_n}(e_n)$ for $n\in\N$ satisfies the desired conditions.

(ii) Let $\varphi_l : C_0((0,1])\otimes M_N \rightarrow A$, $l\in\N$ be a sequence of positive contractions such that $(\varphi_l( \id^q\otimes x))_l = \varphi(\id^q\otimes x)$ in $\Acentral$ for any $q\in\N$, $x\in M_N$. We let $\eij$ be the standard matrix units of $M_N$, and set
\[ \fil = \varphi_l (\id\otimes \eii),\quad \einl = e_n^{1/2} f_{i,l} e_n^{1/2}, \quad {\rm for \ } i=1,2,...,N,\quad l, n\in \N.\]
Since $1-\varphi(\id\otimes 1_N)\in \Iomega$, it follows that 
$ (\fil -\fil^2)_l = (\varphi_l (\id\otimes \eii)(1-\varphi_l (\id\otimes 1_N)))_l \in \Iomega$. Then for any $k$, $m$, $n\in\N$ and $i=1,2,...,N$, we have 
\[ \left(h_m\left(e_n^{1/2} a_k e_n^{1/2}\right) \fil - h_m \left(\einl^{1/2} a_k \einl^{1/2}\right)\right)_l\in\Iomega.\]
Because 
\[ \lim_{l\to\omega}\max_{\tau\in\TA} \left|\tau\left(h_m \left(e_n^{1/2}a_ke_n^{1/2}\right)\left(f_{1,l} -\fil\right)\right)\right| =\lim_{l\to\omega}\max_{\tau}\left|\tau\left(h_m\left( e_n^{1/2} a_k e_n^{1/2}\right)\left(f_{1,l}^2 - \fil^2\right)\right)\right| =0, \]
for any $k$, $m$, $n\in\N$, $i=1,2,...,N$ it follows that 
\[ \lim_{l\to\omega} \min_{\tau\in \TA} \tau \left(h_m \left( \einl^{1/2} a_k \einl^{1/2}\right)\right) =\frac{1}{N} \min_{\tau} \tau\left( h_m (e_n^{1/2} a_k e_n^{1/2})\right).\]

By the assumption in (ii), we inductively obtain a neighborhood $\Omega_L\in\omega$ such that $\Omega_{L+1} \subset \Omega_L$, $\bigcap_{L\in\N}\Omega_L=\emptyset$, and for $n\in\Omega_L$ and $k$, $m\in\{1,2,...,L\}$, 
\[\frac{x_n}{\min_{\tau\in\TA} \tau\left(h_m\left(e_n^{1/2} a_k e_n^{1/2}\right)\right)} < \frac{1}{LN}.\] 
Let $\{b_j\}_{j\in\N}$ be a dense subset of $A$. By the above estimation, for $n\in\Omega_L\setminus \Omega_{L+1}$, we can find $l_n\in\N$ such that $\left\lVert [\filn, e_n] \right\rVert < 1/L$,  $\left\lVert [ \filn, b_j] \right\rVert < 1/L$ for any $i=1,2,...,N$, $j=1,2,...,L$, $\left\lVert\filn \fjln\right\rVert < 1/L$ for $i\neq j$, and  
\[ \frac{x_n}{\min_{\tau}\tau\left( h_m\left( \einln^{1/2} a_k \einln^{1/2}\right)\right)} <\frac{1}{L}\ \forany k, m\in\{1,2,...,L\}\ {\rm and\ } i=1,2,...,N.\]
Define $\ein = \einln$, $n\in\N$. These $(\ein)_n$, $i=1,2,...,N$ satisfy the desired conditions.
\end{proof}

\begin{lemma}\label{LemStrict}
Let $A$ be a unital separable simple \Cs\ with $\TA\neq\emptyset$. Suppose that $A$ has strict comparison and $e_n$, $f_n\in A$, $n\in\N$ and $\{a_k\}_{k\in\N}\subset \Apositiveone$ satisfy the conditions in Definition \ref{DefTI}. Then for any norm one element $a\in A$ there exists a sequence $(r_n)_n$ in $A$ such that 
\[ (r_n^* f_n^{1/2} af_n^{1/2} r_n)_n = (e_n)_n = (r_n^*r_n)_n\quad {\rm in\ }\Aomega.\]
\end{lemma}
\begin{proof}
We may assume that $(e_n)_n \neq 0$ in $\Acentral$. 
For $\varepsilon>0$, it suffices to find $r_n\in A$, $n\in\N$ such that  
\[ \limomega \left\lVert r_n^* f_n^{1/2} af_n^{1/2}r_n -e_n\right\rVert <\varepsilon\quad{\rm and\ } (r_n^*r_n)_n =(e_n)_n\ {\rm in\ }\Aomega.\]
In order to show this, taking a slight perturbation of $a\in A$ we may further assume that $a$ is contained in the given dense subset $\{a_k\}_{k\in\N}$. Now, by the conditions of $e_n$, $f_n$, $\{a_k\}_k$, it follows that 
\[ \limomega\frac{\max_{\tau\in\TA} \tau(e_n)}{\min_{\tau\in\TA} \tau(h_m(f_n^{1/2}af_n^{1/2}))} =0\quad\forany m\in\N.\]
We set $b_n = f_n^{1/2} a f_n^{1/2}$, $n\in\N$.

Since $A$ is simple and $(f_n)_n\in \Acentral$, we see that $\limomega \left\lVert b_n \right\rVert = \limomega \|a\|\cdot\|f_n\|=1$. Then for each $m\in\N$, there is a sequence $(x_{m,n})_n$, $m\in\N$ such that $x_{m,n} >0$ for any $n\in\N$, and $x_{m,n}= \min_{\tau\in\TA} \tau( h_m(b_n))$ for $n$ in a neighborhood of $\omega$. 
Applying Lemma \ref{LemCentral} (i) to $(e_n)_n$ and $x_{m,n}$, we may further assume that 
\[\limomega\frac{\max_{\tau\in \TA} \dt(e_n)}{\min_{\tau\in\TA} \tau(h_m(b_n))}=0\quad\forany m\in\N.\]
This implies that 
\[ \max_{\tau\in\TA} \dt (e_n) < \min_{\tau\in\TA} \tau (h_m (b_n))\leq \min_{\tau\in \TA} \dt (h_m (b_n)),\]
in a neighborhood of $\omega$. 

Let $m\in\N$ be such that $1/m <\varepsilon$. 
 Because of strict comparison of $A$, there exists a sequence $r_n'\in A$, $n\in\N$ such that $(r_n'^* h_m(b_n) r_n')_n =(e_n)_n$ in $\Aomega$. Remark that $\|r_n'\|$ is not necessary bounded. We define $r_n = h_m(b_n)^{1/2} r_n'$, $n\in\N$   which satisfy the desired conditions. Actually, by the definition of $h_m$ it follows that 
\[ (1-b_n) h_m(b_n)\leq \frac{1}{m} h_m (b_n),\] 
which implies 
\[\limomega \left\lVert r_n^*b_n r_n -r_n^*r_n \right\rVert =\limomega\left\lVert r_n'^* h_m(b_n)(1-b_n) r_n' \right\rVert \leq \limomega \frac{1}{m} \left\lVert r_n^*r_n\right\rVert<\varepsilon.\]

\end{proof}

\begin{proof}[Proof of (i) $\Longrightarrow$ (iii) in Proposition \ref{PropTI}]
Note that by the condition of $r_n$ in Lemma \ref{LemStrict}, it automatically follows that $(f_nr_n)_n =(r_n)_n$ in $\Aomega$. Thus, one can follow the argument in the proof of $(ii)\Rightarrow (iii) \Rightarrow (iv) $ of Theorem 1.1 in \cite{MS3}, replacing \cite[Lemma 3.4, Lemma 2.5]{MS3} by  Lemma \ref{LemCentral} (ii) and Lemma \ref{LemStrict}. We omit the detail because the argument is exactly the same.

\end{proof}

\section{$\Js$-absorption of crossed products}
In this section, we prove the main result Theorem \ref{ThmMain}. Although a similar argument can be found in the proof of \cite[Proposition 4.5]{MS2}, (see also \cite[Proposition 4.2]{Sat0}), the main difference in this paper is a small Rohlin tower and property (TI) in the following proposition. 

\begin{proposition}\label{PropSI}
Let $G$ be a countable discrete amenable group,  $A$ a unital separable simple \Cs\ with $\TA \neq\emptyset$, and $\alpha$ an action of $G$ on $A$. If $A$ has property (TI) then the following $\alpha$-equivariant property (SI) holds: 
if sequences $(e_n)_n$ and $(f_n)_n$ of positive contractions in $A$ satisfy $(e_n)_n$, $(f_n)_n\in \Acentral$, $(\alpha_g(e_n))_n =(e_n)_n$, $(\alpha_g(f_n))_n=(f_n)_n$ in $\Acentral$ for any $g\in G$, $(f_n-f_n^2)_n \in \Iomega$, and 
\[\limomega \max_{\tau\in \TA} \tau(e_n) =0,\quad \limomega\min_{\tau\in \TA}\tau (f_n) >0,\]
then there exists a sequence $(s_n)_n$ in $A$ such that $(s_n)_n\in \Acentral$, $(\alpha_g(s_n))_n= (s_n)_n$ for any $g\in G$, and
\[ (s_n^*s_n)_n =(e_n)_n,\quad (f_ns_n)_n=(s_n)_n\ {\rm in\ }\Acentral.\]
\end{proposition}
\begin{proof}
Let $E$ be a finite subset of $G$ and $\varepsilon >0$. Since $G$ is countable it is enough to obtain $(s_n)_n\in \Acentral$ such that $\limomega\left\lVert \alpha_g(s_n)- s_n \right\rVert < \varepsilon$ for $g\in E$, $(s_n^*s_n)_n =(e_n)_n$, $(f_ns_n)_n = (s_n)_n$ in $\Acentral$. 

We let $K$ be the subgroup  of $G$ consisting of all $g\in G$ such that $\alpha_g$ is an inner automorphism. Note that $K$ is a normal subgroup, then the quotient group $G/K$ is amenable.

We set a countable dense subset $\{a_k\}_k$ in $\Apositiveone$. Because of Proposition \ref{PropOuter}, we obtain a net $(\xlambda)_{\lambda\in\Lambda}$ of positive contractions in $A$ such that 

\[ \limlambda \left\lVert [\xlambda, a]\right\rVert =0,\quad  \limlambda \left\lVert \alpha_g(\xlambda) \xlambda\right\rVert =0,\ \forany\ a\in A,{\rm \ and\ }g\in G\setminus K, \]
 and $\xlambda$ strongly converges to a central projection $c\in \Asecond$.  Note that for any $a \in \Apositiveone$, $\xlambda^{1/2} a\xlambda^{1/2}\rightarrow ac$, $\lambda\to\infty$ strongly in $\Asecond$. Then, for any $k, m\in\N$ we have $h_m(\xlambda^{1/2} a_k \xlambda^{1/2}) \to h_m(a_k)c$, $\lambda\to\infty$, strongly in $\Asecond$. Since $A$ is simple, it follows that 
\[ \limlambda \left\lVert h_m (\xlambda^{1/2}a_k\xlambda^{1/2})\right\rVert = \left\lVert h_m(a_k)c \right\rVert = \left\lVert h_m (a_k)\right\rVert =1.\]
Then there exists a sequence $(x_l)_{l\in\N}$ of positive contractions in $A$ such that $ (x_l)_l\in \Acentral$, 
\[ \liml \left\lVert \alpha_g(x_l)x_l\right\rVert =0,\quad \liml \left\lVert h_m (x_l^{1/2} a_k x_l^{1/2}) \right\rVert =1,\]
for any $g\in G\setminus K$ and $k$, $m\in\N$.

For any $N\in\N$, there exists a neighborhood $\Omega_N\in\omega$ such that $h_m(x_l^{1/2} a_k x_l^{1/2})$ is a non-zero positive element in $A$ for any $k$, $m\in\{1,2,...,N\}$, $l\in\Omega_N$. By  \cite[Lemma 2.4]{MS3}, for $k, m\in\{1,2,....,N\}$ and $l\in \Omega_N$ there exists $\alphaklm >0$ such that 
\[ \alphaklm\limomega \min_{\tau\in \TA} \tau(f_n) \leq \limomega\min_{\tau\in \TA} \tau(f_n^{1/2} h_m(x_l^{1/2}a_k x_l^{1/2}) f_n^{1/2}).\]
Let $\yln = x_l^{1/2} f_n x_l^{1/2}$, $l$, $n\in\N$. Since $(f_n-f_n^2)_n\in \Iomega$, it follows that for any $m\in\N$, 
\[\left(h_m \left(\yln^{1/2} a_k \yln^{1/2}\right) - h_m \left(x_l^{1/2} a_k x_l^{1/2}\right)f_n\right)_n\in \Iomega.\]
For this reason, we see that 
\begin{align*}
 0< \alphaklm\limomega\min_{\tau\in\TA} \tau(f_n) &\leq\limomega\min_{\tau} \tau \left(h_m\left(x_l^{1/2} a_k x_l^{1/2}\right)f_n\right) \\
&=\limomega\min_{\tau} \tau\left( h_m \left(\yln^{1/2} a_k \yln^{1/2}\right)\right),
\end{align*}
for any $k, m\in\{1,2,....,N\}$ and $l\in \Omega_N$, which implies 
\[ \limomega \frac{\max_{\tau\in\TA} \tau(e_n)}{\min_{\tau\in\TA}\tau\left(h_m\left(\yln^{1/2}a_k\yln^{1/2}\right)\right)} =0,\quad \forany k, m\in\{1,2,...,N\},\ l\in\Omega_N.\]
Thus it is not so hard to obtain a sequence $(\lNn)_n$ of natural numbers such that $\lNn\to\infty$ when $n\to\omega$, $\left(\left[x_{\lNn}, f_n\right]\right)_n =0$ in $\Acentral$, and  
\[\limomega\frac{\max_{\tau\in\TA} \tau(e_n)}{\min_{\tau\in\TA} \tau\left( h_m\left(y_{\lNn, n}^{1/2}a_ky_{\lNn,n}^{1/2}\right)\right)} =0\quad \forany k, m\in\{1,2,...,N\}.\]

Hence a standard diagonal argument yields a sequence $(l_n)_n$ of natural numbers such that  $l_n\to\infty$, $n\to\omega$, $([x_{l_n}, f_n])_n =0$ in $\Acentral$, (then $y_{l_n,n}\in\Acentral$), and  
\[\limomega\frac{\max_{\tau\in\TA} \tau(e_n)}{\min_{\tau\in\TA} \tau\left(h_m\left(y_{l_n, n}^{1/2}a_ky_{l_n,n}^{1/2}\right)\right)} =0\quad \forany k, m\in\N.\]
Indeed, for $N\in\N$ we inductively obtain a neighborhood $U_N\in\omega$ such that $U_{N+1}\subset U_N$, $\displaystyle \bigcap_{N\in\N}U_N =\emptyset$, and 
for any $n\in U_N$, and $k$, $m\in\{1,2,...,N\}$, 
\[\lNn \geq N,\quad \left\lVert \left[ x_{\lNn}, f_n\right]\right\rVert< \frac{1}{N},\quad \frac{\max_{\tau}\tau(e_n)}{\min_{\tau}\tau\left(h_m\left( y_{\lNn, n}^{1/2} a_k y_{\lNn, n}^{1/2}\right)\right)} < \frac{1}{N}.\]
 We define $l_n=1$ for $n\in\N\setminus U_1$, and $l_n =\lNn$ for $n\in U_N\setminus U_{N+1}$. This $(l_n)_n$ satisfies the above conditions.

Set $y_n =\ylnn$, $n\in\N$. Applying  property (TI) to $(e_n)_n$ and $(y_n)_n$ we obtain a sequence $(r_n)_n$ in $A$ such that $(r_n)_n\in\Acentral$, $(r_n^*r_n)_n = (e_n)_n$, and $(y_nr_n)_n= (r_n)_n$ in $\Acentral$.

We let $H$ denote the amenable group $G/K$ and $\talpha$ the action of $H$ on $\Acentral$ defined by $\talpha_{[g]} (x)=\alpha_g(x)$ for $x\in\Acentral$ and $g\in G$, where $[g]\in H$ is the equivalence class of $g$. Set $[E]=\{[g]\in H \ |\ g\in E\}$.
Let $F\subset H$ be a F\o lner set for $[E]$ and $\varepsilon^2/4 >0$. Fixing a representative $g_f\in G$ for each element in $f\in F$ (i.e., $[g_f]=f$),  we define 
\[ s_n = |F|^{-1/2} \sum_{f\in F} \alpha_{g_f} (r_n),\quad n\in\N.\]
This $(s_n)_n$ satisfies the desired conditions. Actually, by a small Rohlin tower $(x_l)_l$ we have $(\alpha_{g}(y_n)\alpha_{h}(y_n))_n =0$ in $\Acentral$ for $[g]\neq [h]\in H$. Then, for $g\in E$ it follows that 
\begin{align*}
\left\lVert(\alpha_g(s_n)-s_n)_n\right\rVert &= \left\lVert \talpha_{[g]} ((s_n)_n) -(s_n)_n\right\rVert \\
&\leq \frac{1}{|F|^{1/2}}\left\lVert\sum_{f\in [g]F\setminus F} \talpha_{f}((r_n)_n)\right\rVert + \frac{1}{|F|^{1/2}} \left\lVert\sum_{f\in F\setminus [g]F} \talpha_{f}((r_n)_n)\right\rVert \\ 
&\leq \frac{\left\lvert [g]F\setminus F\right\rvert^{1/2}}{|F|^{1/2}}+\frac{\left\lvert F\setminus [g]F\right\rvert^{1/2}}{|F|^{1/2}} < \varepsilon, 
\end{align*}
and that
\begin{align*}
(s_n^*s_n)_n &= \frac{1}{|F|}\sum_{f, h\in F} (\alpha_{g_f}(r_n)^* \alpha_{g_h} (r_n))_n \\
&=\frac{1}{|F|} \sum_{f\in F} (\alpha_{g_f} (r_n^*r_n))_n = (e_n)_n \quad{\rm in\ }\Acentral. 
\end{align*}
Since $(f_n)_n\geq (y_n)_n$, we see that $(f_ns_n)_n =(s_n)_n$ in $\Acentral$.
\end{proof}

\begin{theorem}\label{ThmEmbedding}
Let $G$ be a countable discrete amenable group, $A$ a unital separable simple \Cs\ with an action $\alpha$ of $G$. Assume that $A$ has property (TI) and for any $k\in\N$ there exists an embedding of $C_0((0,1])\otimes M_k$ into the $\alpha$-fixed point subalgebra of $\Acentral$ satisfying the condition in Corollary \ref{CorEmbedding}. 
Then there exists a unital embedding of $\Js$ into the $\alpha$-fixed point subalgebra of $\Acentral$. In particular, Theorem \ref{ThmMain} holds.
\end{theorem}
\begin{proof}
Let $\varphi : C_0((0,1])\otimes M_k \rightarrow \Acentral$ be a $*$-homomorphism obtained in Corollary \ref{CorEmbedding}. Let $\varphi_n : C_0((0,1])\otimes M_k \rightarrow A$, $n\in\N$ be a sequence of completely positive contractions such that $(\varphi_n(a))_n =\varphi(a)$ in $\Acentral$. We denote by $\eij$, $i,j=1,2,...,k$ the standard matrix units of $M_k$. The condition in Corollary \ref{CorEmbedding} induces that 
\[ \limomega\max_{\tau\in\TA} \tau(1-\varphi_n(\id\otimes 1_k))=0, \quad\limomega\min_{\tau\in\TA} \tau(\varphi_n(\id\otimes e_{1,1})) =1/k.\]
We let $e_n$, $f_n$ be positive contractions in $A$ defined by 
\[ e_n = 1-\varphi_n(\id\otimes 1_k),\quad f_n =\varphi_n(\id\otimes e_{1,1}).\]Since the image of $\varphi$ is contained in the $\alpha$-fixed point algebra, especially we have $(\alpha_g(e_n))_n=(e_n)_n$ and $(\alpha_g(f_n))_n=(f_n)_n$ in $\Acentral$ for any $g\in G$.  Note that 
\[(\varphi_n(\id\otimes e_{1,1})-\varphi_n(\id\otimes e_{1,1})^2)_n=(\varphi_n(\id\otimes e_{1,1})(1-\varphi_n(\id\otimes 1_k)))_n\in \Iomega.\]

Applying Proposition \ref{PropSI} to $e_n$ and $f_n$, we obtain a sequence $(s_n)_n$ in $A$ such that $(s_n)_n \in \Acentral$, $(\alpha_g(s_n))_n=(s_n)_n$ for any $g\in G$, $(s_n^*s_n)_n =(e_n)_n$ and $(f_ns_n)_n = (s_n)_n$ in $\Acentral$. We define $c_j =\varphi(\id^{1/2} \otimes e_{1,j})\in \Acentral$, $j=1,2,...,k$ and $s=(s_n)_n\in\Acentral$. Then these elements satisfy 
\[ c_ic_j^* =\delta_{i,j} c_1^2,\quad \sum_{j=1}^k c_j^*c_j + s^*s =1,\quad c_1 s=s\] 
 in the $\alpha$-fixed point subalgebra of $\Acentral$. These conditions construct the dimension drop algebra $I_{k,k+1}$ as the universal \Cs. Recall that $I_{k, k+1}$ was defined by 
\begin{align*}
I_{k,k+1}=\{ f\in C([0,1])\otimes M_k\otimes M_{k+1}\ | \ f(0)\in M_k \otimes \C 1_{k+1},\quad f(1)\in \C 1_{k}\otimes M_{k+1} \},
\end{align*}
and an inductive limit of $I_{k,k+1}$ provides the Jiang-Su algebra  \cite{JS}. Now we obtain a unital $*$-homomorphism from $I_{k,k+1}$ into the $\alpha$-fixed point subalgebra of $\Acentral$. Therefore we can construct a unital embedding of $\Js$ into the $\alpha$-fixed point subalgebra of $\Acentral$, see \cite[Proposition 7.2.2]{Ror0}, \cite[Proposition 2.2]{TW}, \cite[Theorem 4.9]{MS2}.
\end{proof}

As a specific application of Theorem \ref{ThmMain}, we obtain the following results. 
Let $G$ be a countable discrete amenable group, and $B$ be the CAR algebra. We denote by $\sigma$ the Bernoulli shift on $B\cong\bigotimes_{g\in G} B$. 
Then Theorem \ref{ThmMain} implies that $\bigotimes_{G} B \rtimes_{\sigma} G$ absorbs the Jiang-Su algebra tensorially. Furthermore, we can also see a similar result for the Bernoulli shift on the Jiang-Su algebra, this is a generalization of \cite[Corollary 1.4]{Sat0}
\begin{corollary} For a countable discrete amenable group $G$, let $\sigma_{\Js}$ be the Bernoulli shift on $\bigotimes_{g\in G} \Js\cong\Js$. Then the following holds
\[ \left(\bigotimes_G \Js\right)\rtimes_{\sigma_{\Js}} G \otimes \Js \cong \left(\bigotimes_G \Js\right)\rtimes_{\sigma_{\Js}} G.\]
\end{corollary}

\noindent{\it Acknowledgements.}\quad This paper was done during the author's stay at Purdue University. The author would like to thank Marius Dadarlat for his  warm hospitality, and David Kerr and N. Christopher Phillips for helpful conversations. The author is also indebted to Narutaka Ozawa for valuable suggestions on Proposition \ref{PropEmbedCone} and \cite{MK}.

\noindent Yasuhiko Sato \\
Graduate School of Science \\
Kyoto University \\
Sakyo-ku, Kyoto 606-8502\\ 
Japan \\
ysato@math.kyoto-u.ac.jp

\end{document}